\documentclass[12pt,reqno]{amsart}
\usepackage{color}
\usepackage{tikz}
\usepackage{url}

\usepackage{latexsym}
\usepackage{amsfonts,amsmath,mathtools}
\usepackage{graphics}
\usepackage{float}
\usepackage{enumitem}

\numberwithin{equation}{section}

\newtheorem{Theorem}{Theorem}[section]
\newtheorem{Lemma}[Theorem]{Lemma}
\newtheorem{Corollary}[Theorem]{Corollary}
\newtheorem{Proposition}[Theorem]{Proposition}

\newtheorem{Example}[Theorem]{Example}

\newtheorem{Setup}[Theorem]{Setup}

\def\qed{\ifhmode\textqed\fi
	\ifmmode\ifinner\hfill\quad\qedsymbol\else\dispqed\fi\fi}
\def\textqed{\unskip\nobreak\penalty50
	\hskip2em\hbox{}\nobreak\hfill\qedsymbol
	\parfillskip=0pt \finalhyphendemerits=0}
\def\dispqed{\rlap{\qquad\qedsymbol}}

\def\ZZ{\mathbb{Z}}
\def\supp{\textup{supp}}

\def\depth{\textup{depth}}
\def\height{\textup{height}}

\def\Ass{\textup{Ass}}
\def\m{\mathfrak{m}}

\begin{document}
\title{Principal vector-spread Borel ideals}
\author{Marilena Crupi, Antonino Ficarra, Ernesto Lax}

\address{Marilena Crupi, Department of mathematics and computer sciences, physics and earth sciences, University of Messina, Viale Ferdinando Stagno d'Alcontres 31, 98166 Messina, Italy}
\email{mcrupi@unime.it}

\address{Antonino Ficarra, BCAM -- Basque Center for Applied Mathematics, Mazarredo 14, 48009 Bilbao, Basque Country -- Spain, Ikerbasque, Basque Foundation for Science, Plaza Euskadi 5, 48009 Bilbao, Basque Country -- Spain}
\email{aficarra@bcamath.org,\,\,\,\,\,\,\,\,\,\,\,\,\,antficarra@unime.it}

\address{Ernesto Lax, Department of mathematics and computer sciences, physics and earth sciences, University of Messina, Viale Ferdinando Stagno d'Alcontres 31, 98166 Messina, Italy}
\email{erlax@unime.it}


\begin{abstract}
	We study the class of squarefree principal vector-spread Borel ideals. We compute the minimal primary decomposition of these ideals and thereby we prove that they are sequentially Cohen-Macaulay. As the final conclusion of our results, we completely classify the ideals in our class having the property that their ordinary and symbolic powers coincide.
\end{abstract}

\maketitle

\section*{Introduction}\label{sec:intro}
Let \( S = K[x_1, \dots, x_n] \) be the standard graded polynomial ring over a field \( K \). Throughout the paper, monomials \( u \in S \) of degree \( \ell \) will be written as \( u = x_{j_1} x_{j_2} \cdots x_{j_\ell} \) with increasing indices \(1\le\, j_1 \le \cdots \le j_\ell\le n\). 

Following \cite{EHQ, F1}, we consider the notion of \emph{${\bf t}$-spread monomials} for a fixed vector \( {\bf t} = (t_1, \dots, t_{d-1}) \in \ZZ_{\ge 0}^{d-1} \). A monomial \( u = x_{j_1} x_{j_2} \cdots x_{j_\ell} \in S \), with \( \ell \le d \), is said to be \emph{${\bf t}$-spread} if it satisfies the inequality \( j_{k+1} - j_k \ge t_k \) for all \( k = 1, \dots, \ell - 1 \). 

For a monomial ideal \( I \subset S \), let \( G(I) \) denote its unique minimal set of monomial generators. We call \( I \) a \emph{${\bf t}$-spread monomial ideal} if every generator \( u \in G(I) \) is ${\bf t}$-spread. Furthermore, we define a \emph{${\bf t}$-spread strongly stable ideal} to be a ${\bf t}$-spread monomial ideal \( I \subset S \) such that for every ${\bf t}$-spread monomial \( u \in I \), and any indices \( j < i \) such that $x_i$ divides $u$ 
and \( x_j (u / x_i) \) is again ${\bf t}$-spread, it follows that \( x_j (u / x_i) \in I \). This generalizes the classical notions: for \( {\bf t} = \mathbf{0} = (0, \dots, 0) \), one recovers the standard notion of a strongly stable ideal, while the case \( {\bf t} = \mathbf{1} = (1, \dots, 1) \) corresponds to squarefree strongly stable ideals (see, for instance,  \cite{JT}).

Given ${\bf t}$-spread monomials \( u_1, \dots, u_m \in S \), we denote by \( B_{{\bf t}}(u_1, \dots, u_m) \) the smallest ${\bf t}$-spread strongly stable ideal of \( S \) containing all \( u_i \), with respect to the inclusion order. 
The monomials \( u_1, \dots, u_m \) are referred to as the \emph{${\bf t}$-spread Borel generators} of \( B_{{\bf t}}(u_1, \dots, u_m) \). In the case where \( m = 1 \), i.e., \( I = B_{{\bf t}}(u) \), the ideal \( I \) is called a \emph{principal ${\bf t}$-spread Borel ideal}, with ${\bf t}$-spread Borel generator \( u \).

The theory of strongly stable ideals (corresponding to the case \( {\bf t} = \mathbf{0} \)) is by now classical and well-developed; see \cite{JT} for a comprehensive treatment. In characteristic zero, it is known that the generic initial ideal of any homogeneous ideal of \( S \) is strongly stable \cite[Proposition 4.6.2]{JT}. Moreover, the structure and minimal graded free resolutions of such ideals have been described explicitly by Eliahou and Kervaire in their seminal work \cite{EK}, a cornerstone for numerous developments in Combinatorial Commutative Algebra.

Likewise, the theory of vector-spread strongly stable ideals has undergone significant advances in recent years (see, for instance, \cite{AFC, CEL, CF2023, DHQ, NQKR} and the references therein). Furthermore, vector-spread ideals naturally arise in various contexts. They occur as initial ideals of the defining ideals of fiber cones of monomial ideals in two variables \cite{HQS}. They serve as fundamental tools in the study of restricted classes of toric algebras of Veronese type \cite{DHQ} and, moreover, they emerge in the classification of edge ideals whose matching powers are all bi-Cohen-Macaulay \cite{CF2025}.

When the number \( m \) of the \( \mathbf{0} \)-spread Borel generators of a strongly stable ideal \( I \subset S \) becomes too large, certain algebraic properties of \( I \) tend to deteriorate. For instance, in the case \( m = 1 \), {\em i.e.}, when \( I \) is a principal Borel ideal, it is known that the Rees algebra of \( I \) is Koszul \cite{De}. Addressing a question posed by Conca, it was shown in \cite{DFMSS} that the Koszul property of the Rees algebra persists when \( m = 2 \). However, explicit examples demonstrate that this property generally fails for \( m \ge 3 \). These observations indicate that the algebraic behavior of a ${\bf t}$-spread principal Borel ideal is often better behaved than that of an arbitrary ${\bf t}$-spread strongly stable ideal.

The present work aims to investigate the structure of principal vector-spread Borel ideals through the lens of Alexander duality. To this end, we restrict our attention to the squarefree case, and henceforth assume that \( t_i \ge 1 \) for all \( i = 1, \dots, d-1 \).

In Section~\ref{sec1}, we describe explicitly the minimal primary decomposition of a squarefree principal vector-spread Borel ideal (Theorem~\ref{thm:PrimDecom}). The notion of \emph{\(\mathbf{t}\)-spread support} of a \(\mathbf{t}\)-spread monomial plays a key role in the formulation of this result.

In Section~\ref{sec2}, we show that squarefree principal vector-spread Borel ideals are sequentially Cohen-Macaulay (Theorem~\ref{thm:Borel_SCM}). To this end, we prove that their Alexander duals are vertex splittable in the sense of Moradi and Khosh-Ahang \cite{MKA16}. Whether this result holds in the non-squarefree setting is an open question.

Finally, in Section~\ref{sec3}, we classify the squarefree principal vector-spread Borel ideals whose ordinary and symbolic powers coincide (Theorem~\ref{Thm:Bt(u)SP}). Our approach mainly follows the one outlined in \cite{NQKR}. We point out, however, that the proof of the implication (b) $\Rightarrow$ (c) of Theorem~\ref{Thm:Bt(u)SP} differs from that given in \cite[Theorem 5.13]{NQKR}. One can see that the proof in \cite{NQKR} remains valid only when $t_1\ge\dots\ge t_{d-1}$.

The computational packages \cite{FPack, L_pack} have been employed to verify several examples supporting our theoretical results.\vspace*{-0.2em}

\section{Primary decomposition of $B_{\bf t}(u)$}\label{sec1}
In this paper, we focus our attention on squarefree vector-spread principal Borel ideals. The following notation is fixed throughout the paper.
\begin{Setup} \label{setup}
	Let ${\bf t}=(t_1,\dots,t_{d-1})\in\ZZ_{\ge1}^{d-1}$, let $u=x_{j_1}\cdots x_{j_d}\in S$ be a ${\bf t}$-spread monomial with $1\le j_1<\dots<j_{d-1}<j_d=n$ and $d\ge2$ and let $B_{\bf t}(u)$ be the squarefree vector-spread principal Borel ideals with ${\bf t}$-spread Borel generator $u$.
\end{Setup}

The conditions $j_d=n$ and $d\ge2$ are not restrictive. Indeed if $j_d<n$, then $B_{\bf t}(u)$ can be regarded as a monomial ideal of the polynomial ring $K[x_1,\dots,x_{j_d}]$ with fewer variables than $S$. Moreover, if $d=1$ then $B_{\bf t}(u)$ is generated by variables.\smallskip

For an integer $n\ge1$, we set $[n]=\{1,\dots,n\}$. Given a non-empty subset $A\subseteq[n]$, let $P_A=(x_i:\ i\in A)$. For integers $j,k\ge1$, we set $[j,k]=\{h\in\mathbb{N}:\ j\le h\le k\}$. The {\em ${\bf t}$-spread support} of a ${\bf t}$-spread monomial $v=x_{i_1}x_{i_2}\cdots x_{i_{\ell}}\in S$ is defined as
\[
\supp_{\bf t}(v)\ =\ \bigcup_{s=1}^{\ell-1}\,[i_s,i_{s}+(t_s-1)].
\]

Moreover, we set $\supp(v)=\{i_1,\dots,i_\ell\}=\{j:\,x_j \mid v\}$.

\begin{Theorem}\label{thm:PrimDecom}
	The minimal primary decomposition of $B_{\bf t}(u)$ is
	\[
	B_{\bf t}(u)\ =\ (\bigcap_{G\in\mathcal{G}}P_{[n]\setminus G})\cap(\bigcap_{v\in G(B_{\bf t}(u))}P_{[n]\setminus\supp_{\bf t}(v)}),
	\]
	where $\mathcal{G}$ is the family of all subsets of $[n]$ of the form
	\[
	\supp_{\bf t}(x_{\ell_1}x_{\ell_2}\cdots x_{\ell_{s-1}}x_{j_s})\cup[j_s+1,n],
	\]
	such that $x_{\ell_1}x_{\ell_2}\cdots x_{\ell_{s-1}}x_{j_s}\in G(B_{\bf t}(x_{j_1}x_{j_2}\cdots x_{j_{s-1}}x_{j_s}))$, for some $s=1,\dots,d$.
\end{Theorem}
\begin{proof}
	Set $I=B_{\bf t}(u)$. Since $I$ is squarefree, there exists a unique simplicial complex $\Delta$ on the vertex set $[n]$ such that $I=I_\Delta$ is the Stanley-Reisner ideal of $\Delta$. Therefore, by \cite[Lemma 1.5.4]{JT}, the minimal primary decomposition of $I$ is
	\[
	I=\bigcap_{F\in\mathcal{F}(\Delta)}P_{[n]\setminus F},
	\]
	where $\mathcal{F}(\Delta)$ is the set of facets of $\Delta$. Hence, we must prove that
	\begin{equation}\label{facetprimdecomplexinit}
	\mathcal{F}(\Delta)=\mathcal{G}\cup\{\supp_{\bf t}(v):v\in G(B_{\bf t}(u))\}.
	\end{equation}

	Observe that $F\in\mathcal{F}(\Delta)$ if and only if ${\bf x}_F \notin I$ and ${\bf x}_{{F}\cup\{i\}}\in I$, for all $i\in[n]\setminus F$. \smallskip

	We show that each set $F=\supp_{\bf t}(v)$, with $v\in B_{\bf t}(u)$ is a facet of $\Delta$. Let $v=x_{\ell_1}x_{\ell_2}\cdots x_{\ell_d}\in G(B_{\bf t}(u))$. We have
	\[
	F\ =\ \{\ell_1,\ell_1+1,\dots,\ell_1+(t_1-1),\ \dots,\ \ell_{d-1},\ell_{d-1}+1,\dots,\ell_{d-1}+(t_{d-1}-1)\},
	\]
	with $\ell_1\le j_1,\ \ell_2\le j_2,\ \ldots,\ \ell_{d-1}\le j_{d-1}$. Clearly, ${\bf x}_F\notin I$, because, by construction, ${\bf x}_F$ is not a multiple of any ${\bf t}$-spread monomial of degree $d$ of $S$. Hence, $F\in \Delta$.
	
	To prove that $F$ is a facet, it suffices to show that ${\bf x}_{F\cup\{i\}}\in I$, for all $i\in[n]\setminus F$, that is, $F\cup\{i\} \notin \Delta$, for all $i\in[n]\setminus F$. Let $i\in[n]\setminus F$. We have that
	\begin{enumerate}
		\item[-] if $i<\ell_1$, $H=\{i,\ell_1+(t_1-1),\ldots,\ell_{d-1}+(t_{d-1}-1)\}\not\subseteq F$, and ${\bf x}_H\in I=B_{\bf t}(u)$, that is, $H\notin \Delta$. Indeed, $i<\ell_1\le j_1$, $\ell_1+(t_1-1)<\ell_1+t_1\le j_1+t_1\le j_2$, $\dots$, $\ell_{d-1}+(t_{d-1}-1)<\ell_{d-1}+t_{d-1}\le j_{d-1}+t_{d-1}\le j_d$;\\
		\item[-] if $\ell_k+(t_k-1)<i<\ell_{k+1}$, for some $1\le k\le d-2$, then
		\[
		H=\{\ell_1,\ldots,\ell_k,i,\ell_{k+1}+(t_{k+1}-1),\dots,\ell_{d-1}+(t_{d-1}-1)\}\not\subseteq F,
		\]
		and ${\bf x}_H\in I=B_{\bf t}(u)$, arguing as before;\\
		\item[-] if $\ell_{d-1}+(t_{d-1}-1)<i$, then $H=\{\ell_1,\ldots,\ell_{d-1},i\}\not\subseteq F$, and ${\bf x}_H\in I=B_{\bf t}(u)$, in fact $i\le j_d=n$.
	\end{enumerate}
	
	In each case, ${\bf x}_{F\cup\{i\}}$ is a multiple of a ${\bf t}$-spread monomial of degree $d$, belonging to $B_{\bf t}(u)$. Therefore ${\bf x}_{F\cup\{i\}}\in I$ and $F$ is a facet.
	
	Analogously, one proves that any $F\in\mathcal{G}$ is a facet of $\Delta$. Indeed, let $F\in\mathcal{G}$. Then, there exists $s\in\{1,\ldots,d-1\}$ such that
	\[
	F=\supp_{\bf t}(x_{\ell_1}x_{\ell_2}\cdots x_{\ell_{s-1}}x_{j_s})\cup[j_s+1,n],
	\]
	with $x_{\ell_1}x_{\ell_2}\cdots x_{\ell_{s-1}}x_{j_s}\in G(B_{\bf t}(x_{j_1}x_{j_2}\cdots x_{j_{s-1}}x_{j_s}))$. If $s=1$, then $F=[j_1+1,n]$, hence $F\in\mathcal{F}(\Delta)$. Suppose $2\leq s\leq d$, then ${\bf x}_F \notin I$. To prove that $F$ is a facet, we must show that ${\bf x}_{F\cup\{i\}} \in I$, for all $i\in [n]\setminus F$. Let $i\in [n]\setminus F$, we have 
	\begin{enumerate}
		\item[-] if $i<\ell_1$, then
		\[
		H=\{i,\ell_1+(t_1-1),\ldots,\ell_{s-1}+(t_{s-1}-1),j_{s}+(t_{s}-1),\ldots,j_{s}+(t_{s}+\ldots+t_{d-1}-1)\}\not\subseteq F
		\]
		and ${\bf x}_H\in I=B_{\bf t}(u)$, since $i<\ell_1\le j_1$, $\ell_k+(t_k-1)<\ell_k+t_k\le j_k+t_k\le j_{k+1}$ for $2\leq k\leq s-1$, and $j_s+(t_s+\cdots+t_{s+k}-1)<j_s+t_s+ t_{s+1}\cdots+t_{s+k}\leq j_{s+1} +t_{s+1}+\cdots+t_{s+k}\leq\dots\leq j_{s+k}+t_{s+k}\leq j_{s+k+1}$ for $0\leq k\leq d-s-1$;\\
		\item[-] if $\ell_k+(t_k-1)<i<\ell_{k+1}$, for some $1\le k\le s-1$, then
		\[
		H=\{\ell_1,\ldots,\ell_k,i,\ell_{k+1}+(t_{k+1}-1),\dots,\ell_{d-1}+(t_{d-1}-1)\}\not\subseteq F,
		\]
		and ${\bf x}_H\in I=B_{\bf t}(u)$, arguing as before.
	\end{enumerate}
 
	Conversely, we show that the only facets of $\Delta$ are the ones determined before. This is equivalent to prove that for every face $G$ of $\Delta$, there exists $F\in\mathcal{F}(\Delta)$ of the forms described in (\ref{facetprimdecomplexinit}) which contains $G$. 
	
	For this aim, let $G\in\Delta$. We define the following integers
	\begin{align*}
	\ell_1 & = \min(G),\\
	\ell_i & = \min\{r\in G:r\ge\ell_{i-1}+t_{i-1}\}, \ \ \text{for}\ 2\le i\le d.
	\end{align*}
	Suppose $\ell_i\le j_i$, for all $i$. In such case, the sequence $\ell_1<\ell_2<\dots<\ell_k$ can have at most $d-1$ elements. Indeed, if $k\ge d$, then $G\supseteq \{\ell_1, \ell_2, \ldots, \ell_d\}$ with $\ell_i\ge \ell_{i-1}+t_{i-1}$ for $2\le i \le d$. This implies that $\{\ell_1, \ell_2, \ldots, \ell_d\}$ should be a face of $\Delta$, an absurd. In fact, since $\ell_i\le j_i$, for all $i$, then the ${\bf t}$-spread monomial $x_{\ell_1}\cdots x_{\ell_d}$ must belong to $I$. Therefore $k\le d-1$, and $G$ has the following form
	\[G=\{\ell_1,\ell_1+1,\dots,\ell_1+r_1, \dots,\ \ell_{k},\ell_{k}+1,\dots,\ell_{k}+r_k\}\]
	with $0 \le r_i\le t_i-1$, for $i=1, \ldots, k$ and $\ell_i \ge\ell_{i-1}+t_{i-1}$, for $i=2, \ldots, k$. Moreover,
	\begin{align*}
	G\ \subseteq\ F=\{&\ell_1,\ell_1+1,\dots,\ell_1+(t_1-1),\ \dots,\ \ell_{k},\ell_{k}+1,\dots,\ell_{k}+(t_{k}-1),\\
	&j_{k+1},j_{k+1}+1,\dots,j_{k+1}+(t_{k+1}-1),\ \dots,\ j_{d-1},\dots,j_{d-1}+(t_{d-1}-1) \},
	\end{align*}
	and $F$ is a facet as observed before.\smallskip
	
	Suppose now $\ell_i>j_i$ for some $i$. Let $s=\min\{i:\ \ell_i>j_i\}$, then $G$ has the following form
	\[G=\{\ell_1,\ell_1+1,\dots,\ell_1+r_1, \dots,\ \ell_{s-1},\ell_{s-1}+1,\dots,\ell_{s-1}+r_{s-1}\}\]
	with $0 \le r_i\le t_i-1$, for $i=1, \ldots, s-1$ and $\ell_i \ge\ell_{i-1}+t_{i-1}$, for $i=2, \ldots, s-1$. On the other hand,
	\begin{align*}
	G\ \subseteq\ F=\{&\ell_1,\ell_1+1,\dots,\ell_1+(t_1-1),\ \dots,\ \ell_{s-1},\ell_{s-1}+1,\dots,\ell_{s-1}+(t_{s-1}-1),\\ &j_s+1,j_s+2,\dots,n\}=\supp_{\bf t}(x_{\ell_1}x_{\ell_2}\cdots x_{\ell_{s-1}}x_{j_s})\cup[j_s+1,n]\in\mathcal{G}.
	\end{align*}
	This proves the result.
\end{proof}

Theorem \ref{thm:PrimDecom} allows to recover \cite[Corollary 2.4]{CF2023} for the family of squarefree vector-spread principal Borel ideals. For a monomial $u\in S$, set $\min(u)=\min\{j: x_j\mid u\}$.

\begin{Corollary}
	We have $\height\,B_{\bf t}(u) = \min(u)$.
\end{Corollary}
\begin{proof}
Set $I=B_{\bf t}(u)$. By Theorem \ref{thm:PrimDecom}, if $\Delta$ is the simplicial complex whose Stanley-Reisner ideal $I_\Delta$ is $I$, that is $I=I_\Delta$, we have that $\mathcal{F}(\Delta)$ consists of sets of one of the following forms:
	\begin{enumerate}
	\item[-] $F =\{\ell_1,\ell_1+1,\ldots,\ell_1+(t_1-1),\ldots, \ell_{d-1},\ell_{d-1}+1,\ldots,\ell_{d-1}+(t_{d-1}-1)\},$\\
	with $\ell_k \leq j_k$ for $1\leq k\leq d-1$ and $\ell_k - \ell_{k-1}\geq t_{k-1}$ for $2\leq k\leq d-1$;
	\item[]
	\item[-] $F' = \{\ell_1,\ell_1+1,\dots,\ell_1+(t_1-1),\ldots,\ell_{s-1},\ell_{s-1}+1,\ldots,\ell_{s-1}+(t_{s-1}-1),$\\
	$\phantom{..........}j_s+1,j_s+2,\ldots,n\},$\\
	with $1\leq s\leq d-1$, $\ell_k \leq j_k$ for $1\leq k\leq s-1$ and $\ell_k - \ell_{k-1}\geq t_{k-1}$ for $2\leq k\leq s$.
\end{enumerate}
Note that for $s=1$, $F' = [j_1+1, n]$. 
Hence,
\[\height\,I = \min\Big\{n-\sum_{i=1}^{d-1}t_i,\, j_1,\, n-\sum_{i=2}^{s-1}t_i\Big\} =j_1=\min(u).\]

Indeed, since $s\le d-1$ and $u=x_{j_1}x_{j_2}\cdots x_{j_d}$ is a ${\bf t}$-spread monomial, then $j_1\le n - \sum_{i=1}^{d-1}t_i \le n-\sum_{i=2}^{s-1}t_i$. 
\end{proof}

\section{Sequentially Cohen-Macaulayness of $B_{\bf t}(u)$}\label{sec2}
In this section we analyze the sequentially Cohen-Macaulayness of the principal ${\bf t}$-spread Borel ideal $B_{\bf t}(u)$ via the notion of \textit{vertex splitting}.

Recall that a graded $R$-module $M$ is called \textit{sequentially Cohen-Macaulay} 
if there exists a finite filtration of graded $S$-modules
\[
M_0\subset M_1 \subset \cdots \subset M_r=M
\]
such that each $M_i/M_{i-1}$ is Cohen-Macaulay, and the Krull dimensions of the quotients are increasing: $\dim(M_1/M_0)<\dim(M_2/M_1)<\cdots<\dim(M_r/M_{r-1})$.\smallskip

We say that an ideal $I\subset S$ is sequentially Cohen-Macaulay if $S/I$ is a sequentially Cohen-Macaulay $S$-module.\medskip

The goal of this section is to prove the following theorem.
\begin{Theorem}\label{thm:Borel_SCM}
	$B_{\bf t}(u)$ is sequentially Cohen-Macaulay.
\end{Theorem}

The proof of this result requires some preparation.\smallskip

For a graded ideal $I\subset S$, we denote by $I_{\langle i\rangle}$ the ideal generated by the elements of $I$ having degree $i$. We say that $I$ is componentwise linear if $I_{\langle i\rangle}$ has linear resolution for all $i$. By \cite[Theorem 8.2.15]{JT}, monomial ideals with linear quotients are componentwise linear.

Let $I\subset S$ be a squarefree ideal and let $\Delta$ be the unique simplicial complex on the vertex set $[n]$ such that $I_\Delta=I$. The Alexander dual of $I$ is the ideal
\[
I^\vee\ =\ \bigcap_{u\in G(I)}P_{\supp(u)}.
\]

We need the following result \cite[Theorem 8.2.20]{JT}.

\begin{Theorem}\label{thm:dualscmideal}
	Let $I\subset S$ be a squarefree monomial ideal. Then $I$ is sequentially Cohen-Macaulay if and only if $I^\vee$ is componentwise linear.
\end{Theorem}

Following \cite{MKA16}, we say that a monomial ideal $I\subset S$ is \textit{vertex splittable} if it can be obtained by the following recursive procedure:
\begin{enumerate}
	\item[(i)] if $u$ is a monomial and $I=(u)$, $I=(0)$ or $I=S$, then $I$ is vertex splittable;\smallskip
	
	\item[(ii)] if there exists a variable $x_i$ and vertex splittable ideals $I_1\subset S$ and $I_2\subset K[x_1,\dots,x_{i-1},x_{i+1},\dots,x_n]$ such that $I=x_iI_1+I_2$, $I_2\subseteq I_1$ and $G(I)$ is the disjoint union of $G(x_iI_1)$ and $G(I_2)$, then $I$ is vertex splittable.
\end{enumerate}\smallskip

In the case (ii), the decomposition $I=x_iI_1+I_2$ is called a \textit{vertex splitting} of $I$ and $x_i$ is called a \textit{splitting vertex} of $I$. Vertex splittable ideals have linear quotients, and so they are componentwise linear \cite{MKA16}. Recall that a monomial ideal $I\subset S$ has \textit{linear quotients} if the set $G(I)$ can be ordered as $u_1,\dots,u_m$ such that for each $i = 2,\ldots,m$, the ideal $(u_1,\ldots,u_{i-1}):(u_i)$ is generated by variables.\smallskip

The next lemma is crucial.

\begin{Lemma}\label{lem:Dual_VertSplit}
	The ideal $B_{\bf t}(u)^{\vee}$ is vertex splittable. In particular one has
	\begin{equation}\label{eq:Dual_split}
	B_{\bf t}(u)^{\vee} = x_1 B'_{{\bf t}}(u)^{\vee} + B_{\bf t'}'(u/x_{j_1})^{\vee}, 	
	\end{equation}
	where $B'_{{\bf t}}(u)=B_{{\bf t}}(u)\cap K[x_2,\ldots,x_n]$ and ${\bf t'}=(t_2,\ldots,t_{d-1})$.
\end{Lemma}
\begin{proof}
	Let $I=B_{\bf t}(u)=I_{\Delta}$. It follows from \cite[Lemma 2.2]{MKA16} that
	\[
	I^\vee=I_{\Delta}^\vee=x_1I_{\Delta_1}^\vee+I_{\Delta_2}^\vee
	\]
	with $I_{\Delta_2}^\vee\subseteq I_{\Delta_1}^\vee$, and where $\Delta_1$ and $\Delta_2$ are defined as follows
	\[
	\Delta_1=\{F\in\Delta:\ 1\notin F\},\quad \Delta_2=\{F\in\Delta:\ 1\notin F\ \text{and}\ F\cup\{1\}\in\Delta\}.
	\]
	We are going to verify that $I_{\Delta_1}=B_{\bf t}'(u)$ and $I_{\Delta_2}=B_{\bf t'}'(u/x_{j_1})$.
	
	Firstly we prove that $I_{\Delta_1}=B_{\bf t}'(u)$. The inclusion $B_{\bf t}'(u)\subseteq I_{\Delta_1}$ is trivial since the facets of the simplicial complex associated to the squarefree ideal $B_{\bf t}'(u)$ are exactly the facets of $\Delta$ not containing $1$. To prove the other inclusion, let $F\in\Delta$ be a facet and suppose $1\notin\Delta$. By Theorem \ref{thm:PrimDecom}, we have that $F$ can be of the following two forms:
	\begin{enumerate}
		\item[-] $F=\supp_{\bf t}(v)$, where $v\in G(I)$, or
		\item[-] $F=\supp_{\bf t}(x_{\ell_1}x_{\ell_2}\cdots x_{\ell_{s-1}}x_{j_s})\cup [j_s +1 , n]\in\mathcal{G}$, where $s\in\{1,\ldots,d-1\}$ and $x_{\ell_1}x_{\ell_2}\cdots x_{\ell_{s-1}}x_{j_s} \in G(B_{\bf t}(x_{j_1}\cdots x_{j_s}))$.
	\end{enumerate}

	Firstly, let $F=\supp_{\bf t}(v)$ for some $v\in G(I)$. Then $v\in B_{\bf t}'(u)$ since $1\notin F$. Hence ${\bf x}_{[n]\setminus F}\in B'_{{\bf t}}(u)^{\vee}$.
	
	Now, let $F\in\mathcal{G}$ as above. Then $x_{\ell_1}x_{\ell_2}\cdots x_{\ell_{s-1}}x_{j_s} \in B_{\bf t}(x_{j_1}\cdots x_{j_s})\cap K[x_2,\ldots,x_n]$ since $1\notin F$. It follows that ${\bf x}_{[n]\setminus F}\in (B_{\bf t}(x_{j_1}\cdots x_{j_s})\cap K[x_2,\ldots,x_n])^{\vee}$. Note that, since $B_{\bf t}(u)\subset B_{\bf t}(x_{j_1}\cdots x_{j_s})$, one has $(B_{\bf t}(x_{j_1}\cdots x_{j_s})\cap K[x_2,\ldots,x_n])^{\vee}\subset B'_{\bf t}(u)^{\vee}$. Hence, ${\bf x}_{[n]\setminus F}\in B'_{{\bf t}}(u)^{\vee}$ once again. This proves that $I_{\Delta_1}\subseteq B'_{\bf t}(u)$.
	
	With an analogous argument, one can verify that $I_{\Delta_2}=B_{\bf t'}'(u/j_1)$.	
	
	Finally, we prove that $I^\vee$ is vertex splittable. We proceed by double induction on the initial degree $d$ of $I$ and on the number of variables.
	
	If $d=1$, then $I$ is a monomial prime ideal, $I^\vee$ is a monomial principal ideal and the statement follows for any number of variables.
	
	Now, let $d>1$. Notice that $n\ge 1+t_1+\dots+t_{d-1}$, otherwise $I=(0)$. We proceed by induction on $n\ge 1+t_1+\dots+t_{d-1}$. If $n=1+t_1+\dots+t_{d-1}$ then $I=(x_1x_{1+t_1}\cdots x_{1+t_1+\dots+t_{d-1}})$ and so $I^\vee=(x_1,x_{1+t_1},\dots,x_{1+t_1+\dots+t_{d-1}})$ is a monomial prime ideal which is vertex splittable. Now let $n>1+t_1+\dots+t_{d-1}$. By the first part of the proof we have
	\begin{equation}\label{eq:I_split}
	I=x_1B_{\bf t}'(u)+B_{\bf t'}'(u/x_{j_1}).	
	\end{equation}
	Notice that $B_{\bf t}'(u)$ can be regarded as an ideal of $K[x_2,\dots,x_n]$ and so by induction on $n$ is vertex splittable. Moreover, since $\deg(u/x_{j_1})=d-1<d$, by induction on $d$ the ideal $B_{\bf t'}'(u/x_{j_1})$ is also vertex splittable. Since $B_{\bf t'}'(u/x_{j_1})=I_{\Delta_2}\subseteq I_{\Delta_1}=B_{\bf t}'(u)$, it follows that $I$ is also vertex splittable, and this ends the inductive proof.
\end{proof}

Now, we are in the position to prove Theorem \ref{thm:Borel_SCM}.
\begin{proof}[Proof of Theorem \ref{thm:Borel_SCM}]
	Let $I=B_{\bf t}(u)$. By Lemma \ref{lem:Dual_VertSplit}, $I^{\vee}$ is vertex splittable, hence by \cite[Theorem 2.4]{MKA16} it has linear quotients. Therefore, by \cite[Corollary 8.2.22]{JT}, $I^{\vee}$ is componentwise linear. It follows, by \cite[Theorem 8.2.20]{JT} that $I$ is sequentially Cohen-Macaulay as stated.
\end{proof}

Lemma \ref{lem:Dual_VertSplit} combined with \cite[Theorem 2.3]{MKA16} imply immediately

\begin{Corollary}\label{cor:Delta_decomposable}
	Let $B_{\bf t}(u)=I_\Delta$. Then $\Delta$ is vertex decomposable.
\end{Corollary}

The next example illustrates the previous results.
\begin{Example}\label{ex1} \em
	Let ${\bf t}=(2,1)$ and $I=B_{\bf t}(x_2x_5x_8)\subset K[x_1,\ldots,x_8]$.  
	Following Lemma \ref{lem:Dual_VertSplit}, we can write $I=x_1I_2+I_1$ where
	\begin{align*}
		I_1\ &=\ (x_3x_4,\,x_3x_5,\,x_3x_6,\,x_3x_7,\,x_3x_8,\,x_4x_5,\,x_4x_6,\,x_4x_7,\,x_4x_8,\,x_5x_6,\,x_5x_7,\,x_5x_8),\\
		I_2\ &=\ (x_2x_4x_5,\,x_2x_4x_6,\,x_2x_4x_7,\,x_2x_4x_8,\,x_2x_5x_6,\,x_2x_5x_7,\,x_2x_5x_8).
	\end{align*}
    Let ${\bf t}'=(1)$. Then $I^{\vee}=x_1I_1^{\vee}+I_2^{\vee}$ is vertex splittable, with
	\begin{align*}
	I_1^{\vee}\ &=\ B'_{\bf t}(x_2x_5x_8)^\vee\ =\ (x_2,\,x_4x_5,\,x_4x_6x_7x_8,\,x_5x_6x_7x_8),\\
	I_2^{\vee}\ &=\ B'_{{\bf t}'}(x_5x_8)^\vee\ =\ (x_3x_4x_5,\,x_3x_4x_6x_7x_8,\,x_3x_5x_6x_7x_8,\,x_4x_5x_6x_7x_8).
	\end{align*}
    Using the {\em Macaulay2} \cite{GDS} package \texttt{SCMAlgebras} \cite{L_pack} one can verify that $I$ is sequentially Cohen-Macaulay. Alternatively, using the package \texttt{HomologicalShiftIdeals} \cite{FPack} we can check that $I^\vee$ has linear quotients, thus it is componentwise linear.
\end{Example}

\section{Symbolic powers of $B_{\bf t}(u)$}\label{sec3}

Let $I=P_{F_1}\cap\dots\cap P_{F_m}$ be the minimal primary decomposition of a squarefree monomial ideal $I\subset S$. We say that $I$ is \textit{normally torsionfree} if $\Ass(I^k)\subseteq\Ass(I)$ for all $k\ge1$. The $k$th \textit{symbolic power} of $I$ is the monomial ideal defined as
\[
I^{(k)}\ = P_{F_1}^k\cap\dots\cap P_{F_m}^k.
\]

We have $I^k\subseteq I^{(k)}$ for all $k\ge1$. By \cite[Theorem 1.4.6]{JT}, the equality $I^{(k)}=I^k$ holds for all $k\ge1$ if and only if $I$ is normally torsionfree.\smallskip

The goal of this last section is to prove the following characterization.

\begin{Theorem}\label{Thm:Bt(u)SP}
	Let $I=B_{\bf t}(u)$. The following statements are equivalent.
	\begin{enumerate}
		\item[\textup{(a)}] $I$ is normally torsionfree.
		\item[\textup{(b)}] $I^{(k)}=I^k$, for all integers $k\ge1$.
		\item[\textup{(c)}] $j_{i}\le \sum_{s=1}^i t_s$ for all $i=1,\dots,d-1$.
	\end{enumerate}
\end{Theorem}

The proof of this theorem requires some preparation. Let $I\subset S$ be a monomial ideal and let ${\bf a}=(a_1,\dots,a_n)\in\ZZ_{\ge0}^n$. The ${\bf a}$\textit{-restriction} of $I$ is defined as the monomial ideal $I^{\le{\bf a}}\subset S$ with minimal generating set 
\[
G(I^{\le{\bf a}})=\{x_1^{b_1}\cdots x_n^{b_n}\in G(I):\ b_i\le a_i\ \textup{for all}\ i\}.
\]

We recall the following result proved in \cite[Theorem 3.7]{SNQ} (see also \cite[Theorem 1.2]{SNQ2}). Let ${\bf e}_i\in\ZZ_{\ge0}^n$ be the vector with $i$th component equal to 1 and with the other components equal to zero, and set ${\bf 1}={\bf e}_1+\dots+{\bf e}_n=(1,\dots,1)\in\ZZ_{\ge0}^n$.
\begin{Theorem}\label{Thm:useful}
	Let $I\subset S$ be a squarefree monomial ideal. Suppose there exists a squarefree monomial $v\in S$ such that $v\in P\setminus P^2$ for all $P\in\Ass(I)$, and $(I^{\le{\bf 1}-{\bf e}_i})^{(k)}=(I^{\le{\bf 1}-{\bf e}_i})^k$ for all $k\ge1$ and all $i\in\supp(v)$. Then $I^{(k)}=I^k$ for all $k\ge1$.
\end{Theorem}

The following lemma will be needed later.
\begin{Lemma}\label{Lem:FM}
	Let $J\subset S$ be a monomial ideal, let $u\in S$ be a monomial and set $I=uJ$. Then, $I^{(k)}=I^k$ for all $k\ge1$, if and only if, $J^{(k)}=J^k$ for all $k\ge1$
\end{Lemma}
\begin{proof}
	It is shown in \cite[Proposition 6.3]{FM} that $I^{(k)}=u^kJ^{(k)}$ for all $k\ge1$. The assertion now follows at once.
\end{proof}

Let $I\subset S$ be a graded ideal and let $\mathcal{R}(I)=\bigoplus_{k\ge0}I^kt^k$ be its Rees algebra. The \textit{analytic spread} of $I$, denoted by $\ell(I)$, is defined as the Krull dimension of the ring $\mathcal{R}(I)/\m\mathcal{R}(I)=\bigoplus_{k\ge0}(I^k/\m I^k)t^k$, where $\m=(x_1,\dots,x_n)$.
\begin{Proposition}\label{Prop:useful}
	Let $I\subset S$ be a graded ideal such that $\mathcal{R}(I)$ is Cohen-Macaulay. Suppose that $\ell(I)<n$. Then $\m=(x_1,\dots,x_n)\notin\Ass(I^k)$ for all $k\gg0$.
\end{Proposition}
\begin{proof}
	By \cite[Proposition 10.3.2]{JT}, $\lim_{k\rightarrow\infty}\depth\,S/I^k=n-\ell(I)$. Our assumption implies that $\depth\,S/I^k>0$ for all $k\gg0$. Hence $\m\notin\Ass(I^k)$ for all $k\gg0$.
\end{proof}
\begin{Corollary}\label{Cor:useful}
	Let $I=B_{\bf t}(u)\subset S$ be a squarefree principal vector-spread Borel ideal. If $\ell(I)<n$ then $\m\notin\Ass(I^k)$ for all $k\ge1$.
\end{Corollary}
\begin{proof}
	By \cite[Proposition 1]{CF2024}, $I$ is vertex splittable. By \cite[Theorem 3.1(a)-(b)]{M2023} it follows that $\mathcal{R}(I)$ is Cohen-Macaulay and $\Ass(I)\subseteq\Ass(I^2)\subseteq\Ass(I^3)\subseteq\cdots$. Proposition \ref{Prop:useful} and our assumption imply that $\m\notin\Ass(I^k)$ for all $k\gg0$. The previous ascending chain of inclusions yields that $\m\notin\Ass(I^k)$ for all $k\ge1$.
\end{proof}

Let $I\subset S$ be a monomial ideal with $G(I)=\{u_1,\dots,u_m\}$. The \textit{linear relation graph} $\Gamma$ of $I$ is defined as the graph with the edge set 
\[
E(\Gamma)\ =\ \{\{i,j\}: \text{there exist $u_k,u_\ell\in G(I)$ such that $x_iu_k=x_ju_\ell$}\},
\]
and with the vertex set $V(\Gamma)=\bigcup_{\{i,j\}\in E(\Gamma)}\{i,j\}$.

If $I\subset S$ is a monomial ideal generated in degree $d$ having linear relations (e.g. $I$ has linear resolution) then \cite[Lemma 5.2]{DHQ} guarantees that
\begin{equation}\label{eq:ell}
	\ell(I)=r-s+1
\end{equation}
where $r$ is the number of vertices and $s$ is the number of connected components of the linear relation graph of $I$.

Let $I=B_{\bf t}(u)$ as in the Setup \ref{setup}, and define the sets
\[B_r\ =\ \big[\sum_{s=1}^{r-1}t_s+1,\,j_r\big], \quad r=1,\dots,d.\]

Notice that $v=x_{i_1}\cdots x_{i_d} \in G(I)$ if and and only if $i_r\in B_r$, for all $r=1, \ldots, d$, and $i_r - i_{r-1}\ge t_r$ for all $r=2,\dots,d$.

\begin{Proposition}\label{Prop:Gamma}
	Let $I=B_{\bf t}(u)$ and 
	$B_i=[\sum_{s=1}^{i-1}t_s+1,j_i]$ for $i=1,\dots,d$. Let $\Gamma$ be the linear relation graph of $I$. The following statements hold.
	\begin{enumerate}
		\item[\textup{(a)}] We have $|B_i|\ge2$ if and only if $B_i\subseteq V(\Gamma)$. 
		\item[\textup{(b)}] Suppose that $j_i\le\sum_{s=1}^it_s$ for all $i=1,\dots,d-1$. Then $B_p\cap B_q=\emptyset$ for all $1\le p<q\le d$ and $\m\notin\Ass(I^k)$ for all $k\ge1$.
	\end{enumerate}
\end{Proposition}
\begin{proof}
	(a) Suppose that $B_i\subseteq V(\Gamma)$ and assume for a contradiction that $|B_i|=1$, for some $i$. Then $j_i=(\sum_{s=1}^{i-1}t_s)+1$. It follows that each generator of $B_{\bf t}(u)$ is a multiple of $x_{j_i}$ and this implies that $B_i=\{j_i\}\not\subseteq V(\Gamma)$, a contradiction.
	
	Conversely, suppose that $|B_i|\ge2$. We claim  that if $|B_i|\ge2$, then the induced subgraph of $\Gamma$ on $B_i$ is a complete graph. This implies that $B_i\subseteq V(\Gamma)$. 
	
	Let $h_1,h_2\in B_i$ with $h_1<h_2\le j_i$. If $i=1$, then $x_{h_1}(u/x_{j_1}),x_{h_2}(u/x_{j_1})\in G(I)$ since $I$ is ${\bf t}$-spread. Hence $\{h_1,h_2\}\in E(\Gamma)$. If $i=d$, then $vx_{h_1},vx_{h_2}\in G(I)$ where $v=x_1x_{1+t_1}\cdots x_{1+t_1+\dots+t_{d-2}}$. Hence $\{h_1,h_2\}\in E(\Gamma)$. Finally, let $1<i<d$. Then $\sum_{s=1}^{i-1}t_s+1\le h_1<h_2\le j_i$. We have $wx_{h_1},wx_{h_2}\in G(I)$ where $w=x_{j_d}(v/x_{1+t_1+\dots+t_{i-1}})$. Hence $\{h_1,h_2\}\in E(\Gamma)$ as desired.\smallskip
	
	(b) Suppose for a contradiction that $B_p\cap B_q\ne\emptyset$ for some $1\le p<q\le d$. Then $j_p\ge\sum_{s=1}^{q-1}t_s+1>\sum_{s=1}^{p}t_s$. Since by assumption we have $j_p\le\sum_{s=1}^pt_s$, we reach the desired contradiction.
			
	Next, we claim that $\Gamma$ does not contain any edge with one endpoint in $B_p$ and the other endpoint in $B_q$ for some $1\le p<q\le d$. If $|B_p|=1$ or $|B_q|=1$, then (a) implies the assertion. Now, let $|B_p|\ge2$ and $|B_q|\ge2$. Recall that $B_p\cap B_q=\emptyset$ for all $1\le p<q\le d$. Take $h_1\in B_p$ and $h_2\in B_q$ for some $1\le p,q\le d$ with $p\ne q$. Then for any $v\in G(I)$ with $x_{h_1}$ dividing $v$, we have $x_{h_2}(v/x_{h_1})\notin G(I)$.  Indeed, $v$ is not ${\bf t}$-spread because the support of this monomial contains two distinct elements from $B_q$. This yields the assertion.
	
	Now, if $I=(x_1x_{1+t_1}\cdots x_{1+t_1+\dots+t_{d-1}})$ is a principal ideal, there is nothing to prove. Suppose that $I$ is not principal. By Corollary \ref{Cor:useful} it is enough to show that $\ell(I)<n$. By \cite[Corollary 5.3(c)]{F1}, $I$ has linear resolution. Hence, formula \eqref{eq:ell} holds. So it is enough to prove that the linear relation graph $\Gamma$ of $I$ has more than one connected component. Since $I$ is not principal, we have $j_i>(\sum_{s=1}^{i-1}t_s)+1$ for some $i$. Hence $|B_i|\ge2$ for some $i$. By part (a) and the previous claim, it follows that $\Gamma$ is the disjoint union of the complete graphs on the vertex sets $B_i$ for which $|B_i|\ge2$. Hence $\Gamma$ has only one connected component if and only if $|B_d|\geq 2$ and $|B_i|=1$, for all $i=1, \ldots, d-1$. In this case, $\ell(I)=|B_d|<n$. Otherwise, $\Gamma$ has at least two connected components and we have again $\ell(I)<n$, as desired.
\end{proof}

Let $I\subset S$ be a monomial ideal and $P\subset S$ be a monomial prime ideal. The \textit{monomial localization} of $I$ with respect to $P$ is the monomial ideal of the polynomial ring $R=K[x_j:\,x_j\in P]$ defined as $I(P)=\varphi(I)$ where $\varphi:S\rightarrow R$ is the $R$-algebra homomorphism defined by setting $\varphi(x_j)=1$ for all $x_j\notin P$.\medskip

We are now ready to prove Theorem \ref{Thm:Bt(u)SP}.
\begin{proof}
	The equivalence between (a) and (b) follows from \cite[Theorem 1.4.6]{JT}.\smallskip
	
	(b) $\Rightarrow$ (c) Let $I=B_{\bf t}(u)$ and assume that $I^{(k)}=I^k$ for all $k\ge1$. If $d=2$ then the statement follows from \cite[Theorem 5.13]{NQKR}. Hence, we assume that $d\ge3$.
	
	Suppose by contradiction that (c) does not hold. Then, there is at least one index $1\le \ell\le d-1$ such that $j_\ell>\sum_{s=1}^\ell t_s$. Let $\ell$ be the smallest such integer. Let $v=x_{h_1}\cdots x_{h_d}\in G(I)$. Notice that for any $r=1,\dots,\ell-1$ we have
	\begin{equation}\label{eq:hr}
		h_r\in B_r\ \subseteq\ \big[\sum_{s=1}^{r-1}t_s+1,\,\sum_{s=1}^rt_s\big],
	\end{equation}
	and for $r=\ell,\dots,d$ we have $h_r\ge(\sum_{s=1}^{\ell-1}t_s)+1$.
	
	Let $Q=P_{[(\sum_{s=1}^{\ell-1}t_s)+1,n]}$ and $R=K[x_j:x_j\in Q]$. The inclusion \eqref{eq:hr} implies that $I(Q)=B_{\bf s}(v)\cap R$ where ${\bf s}=(t_{\ell},\dots,t_{d-1})$ and $v=u/(x_{j_1}\cdots x_{j_{\ell-1}})$. Hence $I(Q)$ is a ${\bf s}$-spread monomial ideal in the polynomial ring $R$ with $\min(v)=j_\ell\ge 1+\sum_{s=1}^{\ell}t_s$. Moreover, $I(Q)^{(k)}=I(Q)^k$ for all $k\ge1$. Therefore, this reduction shows that we may assume from the very beginning that $\ell=1$ and that $\min(u)=j_1\ge t_1+1$.
	
	To finish the proof, we will show that $I^{(2)}\ne I^2$, reaching the desired contradiction. To this end, we claim that
	\[
	w=(x_1\cdots x_{t_1}x_{t_1+1})(x_{j_2}\cdots x_{j_d})\in I^{(2)}\setminus I^2.
	\] 
	
	Firstly, let us show that $w\in I^{(2)}$. Since $I^{(2)}=\bigcap_{P\in\Ass(I)}P^2$, we must prove that $w\in P^2$ for all $P\in\Ass(I)$.
	
	Let $P\in\Ass(I)$. By Theorem \ref{thm:PrimDecom}, either $P=P_{[n]\setminus\supp_{{\bf t}}(v)}$ for some $v\in G(I)$ or $P=P_{[n]\setminus G}$ for some $G\in\mathcal{G}$, where $\mathcal{G}$ is defined in the statement of Theorem \ref{thm:PrimDecom}.
	
	Suppose $P=P_{[n]\setminus\supp_{{\bf t}}(v)}$. If $\min(v)\ge3$ then $x_1x_2\in P^2$ and so $w\in P^2$. If $\min(v)=2$, then $[2,t_1+1]\subseteq\supp_{\bf t}(v)$. Hence $x_1\in P$ but $x_{1+t_1}\notin P$. Since $x_{1+t_1}x_{j_2}\cdots x_{j_d}\in I\subset P$, $x_{1+t_1}\notin P$ and $P$ is a prime ideal, it follows that $x_{j_s}\in P$ for some $s=2,\dots,d$. Hence $x_1x_{j_s}\in P^2$ and so $w\in P^2$. Finally, if $\min(v)=1$, then $[1,t_1]\subseteq\supp_{\bf t}(v)$. Hence $x_{1+t_1}\in P$ but $x_{1}\notin P$. Since $x_{1}x_{j_2}\cdots x_{j_d}\in I\subset P$ it follows that $x_{j_s}\in P$ for some $s=2,\dots,d$. Hence $x_{1+t_1}x_{j_s}\in P^2$ and so $w\in P^2$.
	
	Suppose now $P=P_{[n]\setminus G}$ with $G\in\mathcal{G}$. Then $G=\supp_{\bf t}(x_{\ell_1}\cdots x_{\ell_{s-1}}x_{j_s})\cup[j_s+1,n]$ where $x_{\ell_1}x_{\ell_2}\cdots x_{\ell_{s-1}}x_{j_s}\in G(B_{\bf t}(x_{j_1}x_{j_2}\cdots x_{j_{s-1}}x_{j_s}))$, for some $s=1,\dots,d$. If $s=1$, then $x_1x_2\in P^2$ and so $w\in P^2$ too. Suppose $s>1$. Then $x_{j_s}\in P$. Moreover, applying the same reasoning as before on the integer $\ell_1$, it follows that $x_1x_2\in P^2$ or $x_1\in P$ or $x_{t_1+1}\in P$. In each possible case, we have $w\in P^2$, as desired.
	
	Hence $w\in I^{(2)}$. Suppose for a contradiction that $w\in I^2$. Then, there would exist squarefree monomials $v_1,v_2\in G(I)$, $g\in S$ such that
	\begin{equation}\label{eq:w}
		w\ = \ (x_1\cdots x_{t_1}x_{t_1+1})(x_{j_2}\cdots x_{j_d})\ =\ v_1v_2g.
	\end{equation}

	Write $v_1=x_{p_1}\cdots x_{p_d}$. Then $p_1\in[1,t_1+1]\cup\{j_2,\dots,j_d\}$. We claim that $p_1\in[1,t_1]$. Suppose, otherwise, that $p_1\in\{j_2,\dots,j_d\}$. Since $p_d>\dots>p_2>p_1$ it would follow from \eqref{eq:w} that $v_1$ divides $x_{j_2}\cdots x_{j_d}$. This is impossible because $v_1$ has degree $d$. Hence $p_1\in[1,t_1+1]$. Since $p_2\ge p_1+t_1$ it follows that $p_2\in\{j_2,\dots,j_d\}$. Therefore $x_{p_2}\cdots x_{p_d}=x_{j_2}\cdots x_{j_d}$. By equation \eqref{eq:w}, it then follows that
	\[
	x_1\cdots x_{p_1-1}x_{p_1+1}\cdots x_{t_1+1}\ =\ v_2g\in I.
	\]
	This is impossible, because $v_2$ is a ${\bf t}$-spread monomial of degree $d\ge3$.\smallskip
	
	(c) $\Rightarrow$ (a) By assumption we have $j_i\le\sum_{s=1}^i t_s$ for all $i=1,\dots,d-1$. We prove by finite induction on $\sum_{i\in\supp(u)}i$ that $I=B_{\bf t}(u)$ is normally torsionfree. 
	
	When $\sum_{i\in\supp(u)}i$ is minimal, then $u=x_{1}x_{1+t_1}\cdots x_{1+t_1+\dots+t_{d-1}}$, $I$ is a principal ideal and so $I$ is normally torsionfree. 
	
	Now suppose that $\sum_{i\in\supp(u)} i$ is not minimal.
	
	Notice that $j_i\ge \sum_{r=1}^{i-1}t_r+1$ for all $r=1,\dots,d$. We may assume that $j_i> \sum_{r=1}^{i-1}t_r+1$ for all $i=1,\dots,d$. Indeed suppose that this is not the case and let $r$ be the biggest integer with this property. Then $B_{\bf t}(u)=(x_{j_1}\cdots x_{j_r})B_{{\bf s}}(v)$ where ${\bf s}=(t_{r+1},\dots,t_{d-1})$ and $v=u/(x_{j_1}\cdots x_{j_r})$, and where $B_{{\bf s}}(v)$ is a ${\bf s}$-spread ideal in the polynomial ring $K[x_{j_r+t_r},\dots,x_n]$ with fewer variables. Then, by Lemma \ref{Lem:FM}, $B_{\bf t}(u)$ is normally torsionfree if and only if $B_{\bf s}(v)$ is such.
	
	Hence we can assume that $j_i> \sum_{r=1}^{i-1}t_r+1$ for all $i=1,\dots,d$. To prove (a) we apply Theorem \ref{Thm:useful}. 
	
	Let $I=B_{\bf t}(u)$ and let $u=x_{j_1}\cdots x_{j_d}$ be the vector-spread Borel generator of $I$. Firstly, let us verify that $u\in P\setminus P^2$ for all $P\in\Ass(I)$. Since $I\subset P$ for all $P\in\Ass(I)$, we only need to verify that $u\notin P^2$ for all $P\in\Ass(I)$. This is equivalent to show that $|G(P)\cap\{x_{j_1},\dots,x_{j_d}\}|\le1$ for all $P\in\Ass(I)$.
	
	Let $P\in\Ass(I)$. By Theorem \ref{thm:PrimDecom}, either $P=P_{[n]\setminus\supp_{\bf t}(v)}$ for some $v\in G(I)$, or else $P=P_{[n]\setminus G}$ for some $G\in\mathcal{G}$. We distinguish the two cases.
	
	Suppose $P=P_{[n]\setminus\supp_{\bf t}(v)}$. We claim that $G(P)\cap\{x_{j_1},\dots,x_{j_d}\}= \{x_{j_d}\}=
	\{x_n\}$ (Setup \ref{setup}) which implies the desired assertion. From the definition of ${\bf t}$-spread support it follows immediately that $x_{j_d}\notin\supp_{\bf t}(v)$ and so $x_{j_d}=x_n\in G(P)$. Now we show that $x_{j_s}\notin G(P)$ for $s=1,\dots,d-1$. Equivalently, we show that $j_s\in\supp_{\bf t}(v)$ for $s=1,\dots,d-1$. Let $v=x_{i_1}\cdots x_{i_d}$. Then $i_s\in B_s=[\sum_{r=1}^{s-1}t_{r}+1,j_s]$, for all $s=1,\dots,d-1$. It follows that
	\[
	i_s+t_s-1\ge \sum_{r=1}^{s-1}t_{r} +1+t_s-1=\sum_{r=1}^st_r\ge j_s\ge i_s.
	\]
	Hence $j_s\in[i_s,i_s+t_s-1]\subseteq\supp_{\bf t}(v)$, as desired.
	
	Suppose $P=P_{[n]\setminus G}$ with $G\in\mathcal{G}$. Then $G=\supp_{\bf t}(x_{\ell_1}\cdots x_{\ell_{s-1}}x_{j_s})\cup[j_s+1,n]$ for some $x_{\ell_1}x_{\ell_2}\cdots x_{\ell_{s-1}}x_{j_s}\in G(B_{\bf t}(x_{j_1}x_{j_2}\cdots x_{j_{s-1}}x_{j_s}))$ and some $s=1,\dots,d$. We claim that $G(P)\cap\{x_{j_1},\dots,x_{j_d}\}=\{x_{j_s}\}$. Notice that $[n]\setminus G=[j_s]\setminus\supp_{\bf t}(x_{\ell_1}\cdots x_{\ell_{s-1}}x_{j_s})$. Then, the same argument as before shows that $G(P)\cap\{x_{j_1},\dots,x_{j_d}\}=\{x_{j_s}\}$.
	
	Now to apply Theorem \ref{Thm:useful} it remains to prove that $(I^{\le{\bf 1}-{\bf e}_{j_i}})^{(k)}=(I^{\le{\bf 1}-{\bf e}_{j_i}})^k$ for all $k\ge1$ and all $i=1,\dots,d$.
	
	Let $1\le i\le d$ and $B_i=[\sum_{s=1}^{i-1}t_s+1,j_i]$, for $i=1,\dots,d$. Our assumption and Proposition \ref{Prop:Gamma} imply that $|B_i|\ge2$ for all $i$ and that the $B_i$'s are pairwise disjoint.  Set $w= x_{j_i-1}(u/x_{j_i})$. If $j_{i}-j_{i-1}>t_i$, then it is clear that  $I^{\le{\bf 1}-{\bf e}_{j_i}}= B_{\bf t}(w)$. 
	Since $j_r\le\sum_{s=1}^rt_s$, for all $r=1,\dots,d$ with $r\ne i$, $j_{i}-1<j_i\le\sum_{s=1}^it_s$ and $\sum_{i\in\supp(w)}i<\sum_{i\in\supp(u)}i$, it follows by induction that $(I^{\le{\bf 1}-{\bf e}_{j_i}})^{(k)}=(I^{\le{\bf 1}-{\bf e}_{j_i}})^k$ for all $k\ge1$, as wanted. Hence, Theorem \ref{Thm:useful} implies that $I^{(k)}=I^k$ for all $k\ge1$.
	
	If otherwise $j_{i}-j_{i-1}=t_i$, then we define the integer
	\[
	\ell\ =\ \max\{m:\ j_p-j_{p-1}=t_{p-1}\ \textup{for all}\ p=m,\dots,i\},
	\]
	and we set
	\[
	v\ =\ \begin{cases}
		\,(\displaystyle\prod_{r=1}^{\ell-2}x_{j_r})(\prod_{r=\ell-1}^{i}x_{j_r-1})(\prod_{r=i+1}^{d}x_{j_r})&\text{if}\ j_i-j_{i-1}=t_{i-1},\\[5pt]
		\,\,x_{j_i-1}(u/x_{j_i})&\text{otherwise}.
	\end{cases}
	\]
	
	We claim that $I^{\le{\bf1}-{\bf e}_{j_i}}=B_{\bf t}(v)$. If $v=x_{j_i-1}(u/x_{j_i})$ this is clear. Now, assume we are in the first case and write $v=x_{h_1}\cdots x_{h_d}$. Notice that $h_r\ge(\sum_{s=1}^{r-1}t_s)+1$ for all $r=1,\dots,d$, because $j_r> \sum_{s=1}^{r-1}t_s+1$ for all $r=1,\dots,d$. Moreover, the definition of $\ell$ implies immediately that $v$ is ${\bf t}$-spread.
	
	Let $w=x_{q_1}\cdots x_{q_d}\in G(B_{\bf t}(v))$. Then $w$ is ${\bf t}$-spread, $\sum_{s=1}^{r-1}t_s +1\le q_r\le h_r\le j_r$ for all $r\in [1,\ell-2]\cup[i+1,d]$, and $\sum_{s=1}^{r-1}t_s+1\le q_r\le h_r\le j_r-1$, for all $r\in [\ell -1, i]$. In particular, $\sum_{s=1}^{i-1}t_s+1\le q_i\le h_i\le j_i-1$. This implies that $x_{j_i}$ does not divide $w$ and so $w\in G(I^{\le{\bf1}-{\bf e}_{j_i}})$.
	
	Conversely, take $w=x_{q_1}\cdots x_{q_d}\in G(I^{\le{\bf1}-{\bf e}_{j_i}})$. Then $w$ is ${\bf t}$-spread with $q_r\in B_r$ for all $r$ and $j_i\notin\supp(w)$. We claim that $q_r\le h_r$ for all $r=1,\dots,d$. For $r\in[1,\ell-2]\cup[i+1,d]$ this is clear. For $r=i$ we have $q_i\in B_i$ and $q_i\ne j_i$. Hence $q_i\le j_i-1=h_i$. Now suppose for a contradiction that $q_{i-1}$ is not less or equal to $j_{i-1}-1=h_{i-1}$. Then, since $q_{i-1}\in B_{i-1}$ we have $q_{i-1}=j_{i-1}$. Since $w$ is ${\bf t}$-spread we have $q_{i}-q_{i-1}-t_{i-1}\ge0$. However, since $q_i\le j_i-1$ and $q_{i-1}=j_{i-1}$ we then have 
	\[
	q_{i}-q_{i-1}-t_{i-1}=j_i-1-j_{i-1}-t_{i-1}=-1
	\]
	by the meaning of $\ell$, and this is absurd. Hence $q_{i-1}\le h_{i-1}$. Proceeding in a similar way we obtain that $q_r\le h_r=j_r-1$ for all $r=\ell-1,\dots,i$. Hence $I^{\le{\bf1}-{\bf e}_{j_i}}\subseteq B_{\bf t}(v)$.
	
	Therefore $I^{\le{\bf 1}-{\bf e}_{j_i}}=B_{\bf t}(v)$. Since $h_r\le j_r\le\sum_{s=1}^rt_s$ for all $r$ and $\sum_{i\in\supp(v)}i<\sum_{i\in\supp(u)}i$, it follows by induction that $(I^{\le{\bf 1}-{\bf e}_{j_i}})^{(k)}=(I^{\le{\bf 1}-{\bf e}_{j_i}})^k$ for all $k\ge1$, as we wanted to show. Finally, Theorem \ref{Thm:useful} implies that $I^{(k)}=I^k$ for all $k\ge1$.
\end{proof}

\textbf{Acknowledgments.} We thank M. Nasernejad for his comments that allowed us to improve the quality of the paper.

All the authors acknowledge support of the GNSAGA of INdAM (Italy). 

A. Ficarra was also supported by 
the Grant JDC2023-051705-I funded by
MICIU/AEI/10.13039/501100011033 and by the FSE+.

\end{document}